\documentclass{amsart}

\RequirePackage{amsthm,amsmath,amsfonts,amssymb}
\RequirePackage[numbers]{natbib}
\RequirePackage[colorlinks,citecolor=blue,urlcolor=blue]{hyperref}
\RequirePackage{graphicx}

\newtheorem{thm}{Theorem}
\newtheorem{lem}[thm]{Lemma}

\newtheorem{prop}[thm]{Proposition}
\theoremstyle{definition}

\newtheorem{ex}[thm]{Example}

\providecommand{\abs}[1]{\lvert#1\rvert}
\providecommand{\Abs}[1]{\Bigl\lvert#1\Bigr\rvert}
\providecommand{\norm}[1]{\lVert#1\rVert}

\begin{document}

\title[Dirichlet sequences]{Kernel based Dirichlet sequences}

\author{Patrizia Berti}
\address{Patrizia Berti, Dipartimento di Matematica Pura ed Applicata ``G. Vitali'', Universit\`a di Modena e Reggio-Emilia, via Campi 213/B, 41100 Modena, Italy} \email{patrizia.berti@unimore.it}

\author{Emanuela Dreassi}
\address{Emanuela Dreassi, Dipartimento di Statistica, Informatica, Applicazioni, Universit\`a di Firenze, viale Morgagni 59, 50134 Firenze, Italy} \email{emanuela.dreassi@unifi.it}

\author{Fabrizio Leisen}
\address{Fabrizio Leisen, School of Mathematical Sciences, University of Nottingham, University Park, Nottingham, NG7 2RD, UK}\email{fabrizio.leisen@gmail.com}

\author{Luca Pratelli}
\address{Luca Pratelli, Accademia Navale, viale Italia 72, 57100 Livorno,
Italy} \email{pratel@mail.dm.unipi.it}

\author{Pietro Rigo}
\address{Pietro Rigo (corresponding author), Dipartimento di Scienze Statistiche ``P. Fortunati'', Universit\`a di Bologna, via delle Belle Arti 41, 40126 Bologna, Italy}
\email{pietro.rigo@unibo.it}


\begin{abstract}
Let $X=(X_1,X_2,\ldots)$ be a sequence of random variables with values in a standard space $(S,\mathcal{B})$. Suppose
\begin{gather*}
X_1\sim\nu\quad\text{and}\quad P\bigl(X_{n+1}\in\cdot\mid X_1,\ldots,X_n\bigr)=\frac{\theta\nu(\cdot)+\sum_{i=1}^nK(X_i)(\cdot)}{n+\theta}\quad\quad\text{a.s.}
\end{gather*}
where $\theta>0$ is a constant, $\nu$ a probability measure on $\mathcal{B}$, and $K$ a random probability measure on $\mathcal{B}$. Then, $X$ is exchangeable whenever $K$ is a regular conditional distribution for $\nu$ given any sub-$\sigma$-field of $\mathcal{B}$. Under this assumption, $X$ enjoys all the main properties of classical Dirichlet sequences, including Sethuraman's representation, conjugacy property, and convergence in total variation of predictive distributions. If $\mu$ is the weak limit of the empirical measures, conditions for $\mu$ to be a.s. discrete, or a.s. non-atomic, or $\mu\ll\nu$ a.s., are provided. Two CLT's are proved as well. The first deals with stable convergence while the second concerns total variation distance.
\end{abstract}

\maketitle

\section{Introduction}\label{intro}

Throughout, $S$ is a Borel subset of a Polish space and $\mathcal{B}$ the Borel $\sigma$-field on $S$. All random elements are defined on a common probability space, say $(\Omega,\mathcal{A},P)$. Moreover,
\begin{gather*}
X=(X_1,X_2,\ldots)
\end{gather*}
is a sequence of random variables with values in $(S,\mathcal{B})$ and
\begin{gather*}
\mathcal{F}_n=\sigma(X_1,\ldots,X_n).
\end{gather*}

We say that $X$ is a {\em Dirichlet sequence}, or a {\em Polya sequence}, if its predictive distributions are of the form
\begin{gather*}
P\bigl(X_{n+1}\in\cdot\mid\mathcal{F}_n\bigr)=\frac{\theta\,P(X_1\in\cdot)+\sum_{i=1}^n\delta_{X_i}(\cdot)}{n+\theta}\quad\quad\text{a.s.}
\end{gather*}
for all $n\ge 1$ and some constant $\theta>0$. The finite measure $\theta\,P(X_1\in\cdot)$ is called the {\em parameter} of $X$. Here and in the sequel, for each $x\in S$, we denote by $\delta_x$ the unit mass at $x$.

Let $\mathcal{L}_0$ be the class of Dirichlet sequences. As it can be guessed from the definition, each element of $\mathcal{L}_0$ is {\em exchangeable}. We recall that $X$ is exchangeable if
\begin{gather*}
\pi(X_1,\ldots,X_n)\sim (X_1,\ldots,X_n)\quad\text{for all }n\ge 2\text{ and all permutations }\pi\text{ of }S^n.
\end{gather*}
A permutation of $S^n$ is meant as a map $\pi:S^n\rightarrow S^n$ of the form
\begin{gather*}
\pi(x_1,\ldots,x_n)=(x_{j_1},\ldots,x_{j_n})\quad\quad\text{for all }(x_1,\ldots,x_n)\in S^n,
\end{gather*}
where $(j_1,\ldots,j_n)$ is a fixed permutation of $(1,\ldots,n)$. An i.i.d. sequence is obviously exchangeable while the converse is not true. However, the distribution of an exchangeable sequence (with values in a standard space) is a mixture of the distributions of i.i.d. sequences; see Subsection \ref{c4rd5}.

\medskip

Since Ferguson, Blackwell and Mac Queen, $\mathcal{L}_0$ played a prevailing role in Bayesian statistics. It was for a long time the basic ingredient of Bayesian nonparametrics. And still today, the Bayesian nonparametrics machinery is greatly affected by $\mathcal{L}_0$ and its developments. In addition, $\mathcal{L}_0$ plays a role in various other settings, including population genetics and species sampling. The literature on $\mathcal{L}_0$ is huge and we do not try to summarize it. Without any claim of being exhaustive, we mention a few seminal papers and recent textbooks: \cite{ANT1974}, \cite{BMQ}, \cite{EW}, \cite{FERG}, \cite{GHOSVANDER}, \cite{HHMW}, \cite{LO}, \cite{PIT1996}, \cite{PITYOR}, \cite{SET}.

\vspace{0.2cm}

The object of this paper is a new class of exchangeable sequences, say $\mathcal{L}$, such that $\mathcal{L}\supset\mathcal{L}_0$. There are essentially two reasons for taking $\mathcal{L}$ into account. First, all main features of $\mathcal{L}_0$ are preserved by $\mathcal{L}$, including the Sethuraman's representation, the conjugacy property and the simple form of predictive distributions. Thus, from the point of view of a Bayesian statistician, $\mathcal{L}$ can be handled as simply as $\mathcal{L}_0$. Second, $\mathcal{L}$ is more flexible than $\mathcal{L}_0$ and allows to model more real situations. For instance, if $X\in\mathcal{L}$, the weak limit of the empirical measures is not forced to be a.s. discrete, but it may be a.s. non-atomic or even a.s. absolutely continuous with respect to a reference measure.

\subsection{Definition of $\mathcal{L}$}\label{v5ty7}

Obviously, the notion of Dirichlet sequence can be extended in various ways. In this paper, for $X$ to be an extended Dirichlet sequence, two conditions are essential. First, $X$ should be exchangeable. Second, the predictive distributions of $X$ should have a known (and possibly simple) structure. Indeed, to define a sequence $X$ via its predictive distributions has various merits. It is technically convenient (see the proof of Theorem \ref{om55d}) and makes the dynamics of $X$ explicit. Furthermore, having the predictive distributions in closed form makes straightforward the Bayesian predictive inference on $X$; see e.g. \cite{BDPR2021} and \cite{HMW}. We also note that, as claimed in \cite{HANPIT}: ``There are very few models for exchangeable sequences $X$ with an explicit prediction rule".

\vspace{0.2cm}

Let $\mathcal{P}$ be the collection of all probability measures on $\mathcal{B}$ and $\mathcal{C}$ the $\sigma$-field over $\mathcal{P}$ generated by the maps $p\mapsto p(A)$ for all $A\in\mathcal{B}$. A {\em kernel} on $(S,\mathcal{B})$ is a measurable map $K:(S,\mathcal{B})\rightarrow (\mathcal{P},\mathcal{C})$. Thus, $K(x)\in\mathcal{P}$ for each $x\in S$ and $x\mapsto K(x)(A)$ is a $\mathcal{B}$-measurable map for fixed $A\in\mathcal{B}$. Here, $K(x)(A)$ denotes the value attached to the event $A$ by the probability measure $K(x)$. (This notation is possibly heavy but suitable for this paper).

\vspace{0.2cm}

A quite natural extension of $\mathcal{L}_0$, among the possible ones, consists in replacing $\delta$ with any kernel $K$ in the predictive distributions of $X$. If $K$ is arbitrary, however, $X$ may fail to be exchangeable.

More precisely, fix $\nu\in\mathcal{P}$, a constant $\theta>0$ and a kernel $K$ on $(S,\mathcal{B})$. By the Ionescu-Tulcea theorem, there is a sequence $X$ such that
\begin{gather}\label{ct6h}
X_1\sim\nu\quad\text{and}\quad P\bigl(X_{n+1}\in\cdot\mid\mathcal{F}_n\bigr)=\frac{\theta\nu(\cdot)+\sum_{i=1}^nK(X_i)(\cdot)}{n+\theta}\quad\quad\text{a.s.}
\end{gather}
for all $n\ge 1$. Generally, however, $X$ is not exchangeable. As an obvious example, take the trivial kernel $K(x)=\nu^*$ for all $x\in S$, where $\nu^*\in\mathcal{P}$ but $\nu^*\neq\nu$. Then, condition \eqref{ct6h} implies that $X_2$ is not distributed as $X_1$.

\vspace{0.2cm}

Our starting point is that, for $X$ to be exchangeable, it suffices condition \eqref{ct6h} and
\begin{gather}\label{v56yh}
K\text{ is a regular conditional distribution (r.c.d.) for }\nu\text{ given }\mathcal{G}
\end{gather}
for some sub-$\sigma$-field $\mathcal{G}\subset\mathcal{B}$. We recall that $K$ is a r.c.d. for $\nu$ given $\mathcal{G}$ if $K(x)\in\mathcal{P}$ for each $x\in S$, the map $x\mapsto K(x)(A)$ is $\mathcal{G}$-measurable for each $A\in\mathcal{B}$, and
\begin{gather*}
\nu(A\cap G)=\int_GK(x)(A)\,\nu(dx)\quad\quad\text{for all }A\in\mathcal{B}\text{ and }G\in\mathcal{G}.
\end{gather*}
Equivalently, $K$ is a r.c.d. for $\nu$ given $\mathcal{G}$ if $K(x)\in\mathcal{P}$ for each $x\in S$ and
\begin{gather*}
K(\cdot)(A)=E_\nu(1_A\mid\mathcal{G}),\quad\nu\text{-a.s.,}\quad\text{for all }A\in\mathcal{B}.
\end{gather*}
Since $(S,\mathcal{B})$ is a standard space, for any sub-$\sigma$-field $\mathcal{G}\subset\mathcal{B}$, a r.c.d. for $\nu$ given $\mathcal{G}$ exists and is $\nu$-essentially unique. See e.g. \cite{BR2007} for more information on r.c.d.'s.

\vspace{0.2cm}

Condition \eqref{v56yh} makes the next definition operational.

\vspace{0.2cm}

Say that $X$ is a {\em kernel based Dirichlet sequence} if it is exchangeable and satisfies condition \eqref{ct6h} for some $\nu\in\mathcal{P}$, some constant $\theta>0$ and some kernel $K$ on $(S,\mathcal{B})$. In particular, $X$ is a kernel based Dirichlet sequence if conditions \eqref{ct6h}-\eqref{v56yh} hold. In the sequel, $\mathcal{L}$ denotes the collection of all $X$ satisfying conditions \eqref{ct6h}-\eqref{v56yh}.

\vspace{0.2cm}

If $X\in\mathcal{L}$ and $\mathcal{G}=\mathcal{B}$, then $K=\delta$ and $X\in\mathcal{L}_0$. At the opposite extreme, if $\mathcal{G}=\{\emptyset,S\}$, then $K(x)=\nu$ for $\nu$-almost all $x\in S$ and $X$ is i.i.d. Various other examples come soon to the fore. The following are from \cite{BDPR2021} (even if, when writing \cite{BDPR2021}, we didn't know yet that $X$ is exchangeable).

\vspace{0.2cm}

\begin{ex}\label{w4f}
Let $\mathcal{G}=\sigma(\mathcal{H})$, where $\mathcal{H}\subset\mathcal{B}$ is a countable partition of $S$ such that $\nu(H)>0$ for all $H\in\mathcal{H}$. A r.c.d. for $\nu$ given $\mathcal{G}$ is
\begin{gather*}
K(x)=\sum_{H\in\mathcal{H}}1_H(x)\,\nu(\cdot\mid H)=\nu\bigl[\cdot\mid H(x)\bigr]
\end{gather*}
where $H(x)$ denotes the only $H\in\mathcal{H}$ such that $x\in H$. Therefore, $X\in\mathcal{L}$ whenever
\begin{gather*}
X_1\sim\nu\quad\text{and}\quad P\bigl(X_{n+1}\in\cdot\mid\mathcal{F}_n\bigr)=\frac{\theta\nu(\cdot)+\sum_{i=1}^n\nu\bigl[\cdot\mid H(X_i)\bigr]}{n+\theta}\quad\quad\text{a.s.}
\end{gather*}
Note that
\begin{gather*}
P\bigl(X_{n+1}\in\cdot\mid\mathcal{F}_n\bigr)\ll\nu(\cdot)\quad\quad\text{a.s.}
\end{gather*}
This fact highlights a stricking difference between $\mathcal{L}$ and $\mathcal{L}_0$. In this example, if $\nu$ is non-atomic, the probability distributions of $X$ and $Y$ are singular for any $Y\in\mathcal{L}_0$.
\end{ex}

\vspace{0.2cm}

\begin{ex}\label{w2wi9}
Let $S=\mathbb{R}^2$ and $\mathcal{G}=\sigma(f)$ where $f(u,v)=u$ for all $(u,v)\in\mathbb{R}^2$. Let $\mathcal{B}_0$ be the Borel $\sigma$-field on $\mathbb{R}$ and $\mathcal{N}(u,1)$ the Gaussian law on $\mathcal{B}_0$ with mean $u$ and variance 1. Fix a probability measure $r$ on $\mathcal{B}_0$ and define
\begin{gather*}
\nu(A\times B)=\int_A\mathcal{N}(u,1)(B)\,r(du)\quad\quad\text{for all }A,\,B\in\mathcal{B}_0
\end{gather*}
where $\mathcal{N}(u,1)(B)$ denotes the value attached to $B$ by $\mathcal{N}(u,1)$. Then, a r.c.d. for $\nu$ given $\mathcal{G}$ is
\begin{gather*}
K(u,v)=\delta_u\times\mathcal{N}(u,1)\quad\quad\text{for all }(u,v)\in\mathbb{R}^2.
\end{gather*}
Hence, letting $X_i=(U_i,V_i)$, one obtains $X\in\mathcal{L}$ provided $(U_1,V_1)\sim\nu$ and
\begin{gather*}
P\bigl(U_{n+1}\in A,\,V_{n+1}\in B\mid\mathcal{F}_n\bigr)=\frac{\theta\nu(A\times B)+\sum_{i=1}^n1_A(U_i)\,\mathcal{N}(U_i,1)(B)}{n+\theta}\quad\quad\text{a.s.}
\end{gather*}
\end{ex}

\vspace{0.2cm}

\begin{ex}\label{u7z2}
Let $f:S\rightarrow S$ be a measurable map. If $\nu$ is $f$-invariant, that is $\nu=\nu\circ f^{-1}$, it may be reasonable to take
\begin{gather*}
\mathcal{G}=\bigl\{A\in\mathcal{B}:f^{-1}(A)=A\bigr\}.
\end{gather*}
As a trivial example, if $S=\mathbb{R}$, $f(x)=-x$ and $\nu$ is symmetric, then
\begin{gather*}
K(x)=\frac{\delta_x+\delta_{-x}}{2}
\end{gather*}
is a r.c.d. for $\nu$ given $\mathcal{G}$. Hence, $X\in\mathcal{L}$ whenever $X_1\sim\nu$ and
\begin{gather*}
P\bigl(X_{n+1}\in\cdot\mid\mathcal{F}_n\bigr)=\frac{2\,\theta\nu+\sum_{i=1}^n(\delta_{X_i}+\delta_{-X_i})}{2\,(n+\theta)}\quad\quad\text{a.s.}
\end{gather*}
This example is related to \cite{BDPR2021}, \cite{DALAL} and \cite{HOZA}. We will take up it again in forthcoming Example \ref{w35b8uh6}.
\end{ex}

\subsection{Sethuraman's representation and conjugacy for $\mathcal{L}_0$}\label{c4rd5}

Before going on, a few basic properties of $\mathcal{L}_0$ are to be recalled.

\vspace{0.2cm}

A {\em random probability measure} on $(S,\mathcal{B})$ is a measurable map $\mu:(\Omega,\mathcal{A})\rightarrow (\mathcal{P},\mathcal{C})$.

\vspace{0.2cm}

Let $X$ be exchangeable. Since $(S,\mathcal{B})$ is a standard space, there is a random probability measure $\mu$ on $(S,\mathcal{B})$ such that
\begin{gather*}
\mu(A)\overset{a.s.}=\lim_n\frac{1}{n}\,\sum_{i=1}^n1_A(X_i)\overset{a.s.}=\lim_nP\bigl(X_{n+1}\in A\mid\mathcal{F}_n\bigr)
\end{gather*}
for each fixed $A\in\mathcal{B}$. Moreover, $X$ is i.i.d. conditionally on $\mu$, in the sense that
\begin{gather*}
P\bigl(X\in B\mid\mu\bigr)=\mu^\infty(B)\quad\quad\text{a.s. for all }B\in\mathcal{B}^\infty
\end{gather*}
where $\mu^\infty=\mu\times\mu\times\ldots$; see e.g. \cite[p. 2090]{BPR2013}.

\vspace{0.2cm}

Suppose now that $X\in\mathcal{L}_0$ and define
\begin{gather*}
\mathcal{D}(C)=P(\mu\in C)\quad\quad\text{for all }C\in\mathcal{C}.
\end{gather*}
Such a $\mathcal{D}$ is a probability measure on $\mathcal{C}$, called the {\em Dirichlet prior}, and admits the following representation. Define a random probability measure $\mu^*$ on $(S,\mathcal{B})$ as
\begin{gather*}
\mu^*=\sum_jV_j\,\delta_{Z_j},
\end{gather*}
where $(Z_j)$ and $(V_j)$ are independent sequences, $(Z_j)$ is i.i.d. with $Z_1\sim\nu$, and $(V_j)$ has the stick-breaking distribution with parameter $\theta$; see Section \ref{m98f4}. Then,
\begin{gather*}
\mathcal{D}(C)=P(\mu^*\in C)\quad\quad\text{for all }C\in\mathcal{C}.
\end{gather*}
Thus, $\mathcal{D}$ can be also regarded as the probability distribution of $\mu^*$. This fact, proved by Sethuraman \cite{SET}, is fundamental in applications; see e.g. \cite{FLP}.

\vspace{0.2cm}

Finally, we recall the conjugacy property of $\mathcal{L}_0$. Write $\mathcal{D}(\lambda)$ (instead of $\mathcal{D}$) if $X\in\mathcal{L}_0$ has parameter $\lambda$. In this notation, if $X$ has parameter $\theta\nu$, then
\begin{gather*}
P(\mu\in C\mid\mathcal{F}_n)=\mathcal{D}\Bigl(\theta\nu+\sum_{i=1}^n\delta_{X_i}\Bigr)(C)\quad\quad\text{a.s. for all }C\in\mathcal{C}\text{ and }n\ge 1.
\end{gather*}
Roughly speaking, the posterior distribution of $\mu$ given $(X_1,\ldots,X_n)$ is still of the Dirichlet type but the parameter turns into $\theta\nu+\sum_{i=1}^n\delta_{X_i}$. Once again, this fact plays a basic role in applications.

\subsection{Our contribution} As claimed above, this paper aims to introduce and investigate the class $\mathcal{L}$.

\vspace{0.2cm}

Our first result is that conditions \eqref{ct6h}-\eqref{v56yh} suffice for exchangeability of $X$. Thus, each $X\in\mathcal{L}$ is a kernel based Dirichlet sequence, as defined in Subsection \ref{v5ty7}.

\vspace{0.2cm}

The next step is to develop some theory for $\mathcal{L}$. The obvious hope is that, at least to a certain extent, such a theory is parallel to that of $\mathcal{L}_0$. This is exactly the case. Essentially all main results concerning $\mathcal{L}_0$ extend nicely to $\mathcal{L}$. To illustrate, we assume $X\in\mathcal{L}$ and we mention a few facts.

\vspace{0.2cm}

\begin{itemize}

\item Up to replacing $\delta$ with $K$, the Sethuraman's representation remains exactly the same. Precisely, $P(\mu\in C)=P(\mu^*\in C)$ for all $C\in\mathcal{C}$, where
\begin{gather*}
\mu^*=\sum_jV_j\,K(Z_j)
\end{gather*}
and $(V_j)$ and $(Z_j)$ are as in Subsection \ref{c4rd5}.

\vspace{0.2cm}

\item The predictive distributions converge in total variation, that is
\begin{gather*}
\sup_{A\in\mathcal{B}}\,\Abs{P\bigl(X_{n+1}\in A\mid\mathcal{F}_n\bigr)-\mu(A)}\overset{a.s.}\longrightarrow 0\quad\quad\text{as }n\rightarrow\infty.
\end{gather*}

\vspace{0.2cm}

\item If $X\in\mathcal{L}_0$, it is well known that $\mu$ is a.s. discrete. This result extends to $\mathcal{L}$ as follows. Denote by $D_1$, $D_2$, $D_3$ the collections of elements of $\mathcal{P}$ which are, respectively, discrete, non-atomic, or absolutely continuous with respect to $\nu$. Then, for each  $1\le j\le 3$,
\begin{gather*}
P(\mu\in D_j)=1\quad\Leftrightarrow\quad K(x)\in D_j\text{ for }\nu\text{-almost all }x\in S.
\end{gather*}
Since $\delta_x\in D_1$ for all $x\in S$, the classical result is recovered. But now, with a suitable $K$, one obtains $P(\mu\in D_2)=1$ or $P(\mu\in D_3)=1$. This fact may be useful in applications.

\vspace{0.2cm}

\item The conjugacy property of $\mathcal{L}_0$ is still available. For each $n\ge 1$, let
\begin{gather*}
V^{(n)}=\bigl(V_j^{(n)}:j\ge 1\bigr)\quad\text{and}\quad Z^{(n)}=\bigl(Z_j^{(n)}:j\ge 1\bigr)
\end{gather*}
be two sequences such that

\vspace{0.2cm}

(i) $V^{(n)}$ and $Z^{(n)}$ are conditionally independent given $\mathcal{F}_n$;

(ii) $V^{(n)}$ has the stick-breaking distribution, with parameter $n+\theta$, conditionally on $\mathcal{F}_n$;

(iii) $Z^{(n)}$ is i.i.d., conditionally on $\mathcal{F}_n$, with
\begin{gather*}
P(Z_1^{(n)}\in\cdot\mid\mathcal{F}_n)=P\bigl(X_{n+1}\in\cdot\mid\mathcal{F}_n\bigr)=\frac{\theta\nu(\cdot)+\sum_{i=1}^nK(X_i)(\cdot)}{n+\theta}\quad\quad\text{a.s.}
\end{gather*}

\noindent Then,
\begin{gather*}
P(\mu\in\cdot\mid\mathcal{F}_n)=P(\mu^*_n\in\cdot\mid\mathcal{F}_n)
\end{gather*}
where
\begin{gather*}
\mu_n^*=\sum_jV_j^{(n)}\,K\bigl(Z_j^{(n)}\bigr).
\end{gather*}
Again, if $K=\delta$, this result reduces to the classical one.

\vspace{0.2cm}

\item A stable CLT holds true. Let $S=\mathbb{R}^p$ and $\int\norm{x}^2\,\nu(dx)<\infty$, where $\norm{\cdot}$ is the Euclidean norm. Suppose that $K$ has mean 0, in the sense that
\begin{gather*}
\int y_i\,K(x)(dy)=0\quad\quad\text{for all }x\in\mathbb{R}^p\text{ and }i=1,\ldots,p
\end{gather*}
where $y_i$ denotes the $i$-th coordinate of a point $y\in\mathbb{R}^p$. Then, $n^{-1/2}\sum_{i=1}^nX_i$ converges stably (in particular, in distribution) to the Gaussian kernel $\mathcal{N}_p(0,\Sigma)$, where $\Sigma$ is the (random) covariance matrix
\begin{gather*}
\Sigma=\left(\,\int y_i\,y_j\,\mu(dy):1\le i,\,j\le p\right).
\end{gather*}
Moreover, under some additional conditions, $n^{-1/2}\sum_{i=1}^nX_i$ converges in total variation as well.

\end{itemize}

\vspace{0.2cm}

This is a brief summary of our main results. Before closing the introduction, however, two remarks are in order.

\vspace{0.2cm}

First, to prove such results, we often exploit the fact that
\begin{gather}\label{y93eh5tg}
\bigl(K(X_n):n\ge 1\bigr)\quad\text{is a classical Dirichlet sequence with values in }(\mathcal{P},\mathcal{C}).
\end{gather}
Condition \eqref{y93eh5tg} is not surprising. We give a simple proof of it, based on predictive distributions, but condition \eqref{y93eh5tg} could be also obtained via some known results on $\mathcal{L}_0$.

\vspace{0.2cm}

Second, the above results are potentially useful in Bayesian nonparametrics. Define in fact
\begin{gather*}
\Pi(C)=P(\mu\in C)=P(\mu^*\in C)\quad\quad\text{for all }C\in\mathcal{C}.
\end{gather*}
Such a $\Pi$ is a new prior to be used in Bayesian nonparametrics. In real problems, working with $\Pi$ is as simple as working with the classical Dirichlet prior $\mathcal{D}$. In both cases, the posterior can be easily evaluated. Unlike $\mathcal{D}$, however, $\Pi$ can be chosen such that $\Pi(C)=1$ for some meaningful sets $C$ of probability measures. For instance, $C=D_j$ with $D_j$ defined as above for $j=1,2,3$. Or else, $C$ the set of invariant probability measures under a countable class of measurable transformations; see forthcoming Example \ref{w35b8uh6}. Finally, just because of its definition, $\mathcal{L}$ is particularly suitable in Bayesian predictive inference. And predicting future observations is one of the main tasks  of Bayesian nonparametrics.

\section{Preliminaries}\label{m98f4}

For all $\lambda\in\mathcal{P}$ and bounded measurable $f:S\rightarrow\mathbb{R}$, the notation $\lambda(f)$ stands for $\lambda(f)=\int f\,d\lambda$. Moreover, $\mathcal{N}_p(0,\Sigma)$ denotes the $p$-dimensional Gaussian law (on the Borel $\sigma$-field of $\mathbb{R}^p$) with mean 0 and covariance matrix $\Sigma$.

\vspace{0.2cm}

Let $\theta>0$ be a constant, $(W_n)$ an i.i.d. sequence with $W_1\sim$ beta$(1,\theta)$ and
\begin{gather*}
T_1=W_1,\quad T_n=W_n\prod_{i=1}^{n-1}(1-W_i)\text{ for }n>1.
\end{gather*}
A sequence $(V_n)$ of real random variables has the {\em stick-breaking distribution with parameter} $\theta$ if $(V_n)\sim (T_n)$. Note that $V_n>0$ for all $n$ and $\sum_nV_n=1$ a.s.

\vspace{0.2cm}

{\em Stable convergence} is a strong form of convergence in distribution. Let $N$ be a random probability measure on $(S,\mathcal{B})$. Then, $X_n$ {\em converges to $N$ stably} if
\begin{gather*}
E\bigl[N(f)\mid H\bigr]=\lim_nE\bigl[f(X_n)\mid H\bigr]
\end{gather*}
for all bounded continuous $f:S\rightarrow\mathbb{R}$ and all $H\in\mathcal{A}$ with $P(H)>0$. In particular, $X_n$ converges in distribution to the probability measure $A\mapsto E\bigl[N(A)\bigr]$.

\vspace{0.2cm}

We next report an useful characterization of exchangeability due to \cite{FLR2000}; see also \cite{BPR2012} and \cite{BDPR2021}. Let $\mathcal{F}_0=\{\emptyset,\Omega\}$ be the trivial $\sigma$-field and
\begin{gather*}
\sigma_n(x)=P\bigl[X_{n+1}\in\cdot\mid (X_1,\ldots,X_n)=x\bigr]\quad\quad\text{for all }x\in S^n.
\end{gather*}

\begin{thm} {\bf (\cite[Theorem 3.1]{FLR2000}).}\label{markov} The sequence $X$ is exchangeable if and only if
\begin{gather*}
P\bigl[(X_{n+1},X_{n+2})\in\cdot\mid\mathcal{F}_n\bigr]=P\bigl[(X_{n+2},X_{n+1})\in\cdot\mid\mathcal{F}_n\bigr]\quad\quad\text{a.s.}
\end{gather*}
for all $n\ge 0$ and
\begin{gather*}
\sigma_n(x)=\sigma_n(\pi(x))
\end{gather*}
for all $n\ge 2$, all permutations $\pi$ of $S^n$, and almost all $x\in S^n$. (Here, ``almost all" is with respect to the marginal distribution of $(X_1,\ldots,X_n)$).
\end{thm}

We conclude this section with two technical lemmas. Let
\begin{gather*}
\sigma(K)=\bigl\{\bigl\{x\in S:K(x)\in C\bigr\}:C\in\mathcal{C}\bigr\}
\end{gather*}
be the $\sigma$-field over $S$ generated by the kernel $K$.

\begin{lem}\label{w3d5z5n9i}\textbf{(Lemma 10 of \cite{BR2007}).} Under condition \eqref{v56yh}, there is a set $F\in\sigma(K)$ such that $\nu(F)=1$ and
\begin{gather*}
K(x)(B)=\delta_x(B)\quad\quad\text{for all }B\in\sigma(K)\text{ and }x\in F.
\end{gather*}
\end{lem}

\begin{proof} This is basically \cite[Lem. 10]{BR2007} but we give a proof to make the paper self-contained. The atoms of the $\sigma$-field $\sigma(K)$ are sets of the form
\begin{gather*}
B(x)=\bigl\{y\in S:K(y)=K(x)\bigr\}\quad\quad\text{for all }x\in S.
\end{gather*}
Hence, each $B\in\sigma(K)$ can be written as
\begin{gather*}
B=\bigcup_{x\in B}B(x).
\end{gather*}
Moreover, by \cite[Lem. 10]{BR2007}, there is a set $F\in\sigma(K)$ such that $\nu(F)=1$ and
\begin{gather*}
K(x)\bigl(B(x)\bigr)=1\quad\quad\text{for all }x\in F.
\end{gather*}
Having noted these facts, fix $x\in F$ and $B\in\sigma(K)$. If $x\in B$, then
\begin{gather*}
K(x)(B)\ge K(x)\bigl(B(x)\bigr)=1.
\end{gather*}
If $x\notin B$, since $B^c\in\sigma(K)$, then $K(x)(B)=1-K(x)(B^c)=0$. Hence, $K(x)(B)=\delta_x(B)$.
\end{proof}

\medskip

\begin{lem}\label{w5v7uj}
Under condition \eqref{v56yh}, there is a set $F\in\sigma(K)$ such that $\nu(F)=1$ and
\begin{gather*}
\int_AK(y)(B)\,K(x)(dy)=K(x)(A)\,K(x)(B)\quad\quad\text{for all }x\in F\text{ and }A,\,B\in\mathcal{B}.
\end{gather*}
Moreover,
\begin{gather*}
\int_AK(y)(B)\,\nu(dy)=\int_BK(y)(A)\,\nu(dy)\quad\quad\text{for all }A,\,B\in\mathcal{B}.
\end{gather*}
\end{lem}

\begin{proof}
Let $F$ be as in Lemma \ref{w3d5z5n9i}. Fix $x\in F$ and $A,\,B\in\mathcal{B}$. Define
\begin{gather*}
G=\bigl\{y\in S:K(y)(B)=K(x)(B)\bigr\}
\end{gather*}
and note that $x\in G$ and $G\in\sigma(K)$. Since $x\in G$, then $\delta_x(G)=1$. Since $G\in\sigma(K)$ and $x\in F$, Lemma \ref{w3d5z5n9i} implies
\begin{gather*}
K(x)(G)=\delta_x(G)=1.
\end{gather*}
Therefore,
\begin{gather*}
\int_AK(y)(B)\,K(x)(dy)=K(x)(B)\,\int_A\,K(x)(dy)=K(x)(A)\,K(x)(B).
\end{gather*}
Finally,
\begin{gather*}
\int_AK(y)(B)\,\nu(dy)=\int_A E_\nu(1_B\mid\mathcal{G})\,d\nu=\int_B E_\nu(1_A\mid\mathcal{G})\,d\nu=\int_BK(y)(A)\,\nu(dy).
\end{gather*}
\end{proof}

\section{Results}\label{r1e2s3}

Recall that $\mathcal{L}$ is the class of sequences satisfying conditions \eqref{ct6h}-\eqref{v56yh} for some $\nu\in\mathcal{P}$ and some constant $\theta>0$. In this section, $X\in\mathcal{L}$ and $\mu$ is a random probability measure on $(S,\mathcal{B})$ such that
\begin{gather*}
\mu(A)\overset{a.s.}=\lim_n\frac{1}{n}\,\sum_{i=1}^n1_A(X_i)\overset{a.s.}=\lim_nP\bigl(X_{n+1}\in A\mid\mathcal{F}_n\bigr)\quad\quad\text{for all }A\in\mathcal{B}.
\end{gather*}
Existence of $\mu$ depends on $X$ is exchangeable and $(S,\mathcal{B})$ is a standard space; see Subsection \ref{c4rd5}.

\vspace{0.2cm}

Our starting point is the following.

\begin{thm}\label{bn7}
Under condition \eqref{ct6h}, $X$ is exchangeable if and only if
\begin{gather}\label{d5yh8q}
\int_AK(y)(B)\,\nu(dy)=\int_BK(y)(A)\,\nu(dy)
\end{gather}
and
\begin{gather}\label{q2a3x}
\int_AK(y)(B)\,K(x)(dy)=\int_BK(y)(A)\,K(x)(dy)
\end{gather}
for all $A,\,B\in\mathcal{B}$ and $\nu$-almost all $x\in S$. In particular, $X$ is exchangeable whenever $X\in\mathcal{L}$ (because of Lemma \ref{w5v7uj}).
\end{thm}

\begin{proof}
For all $A,\,B\in\mathcal{B}$, condition \eqref{ct6h} implies
\begin{gather*}
P(X_1\in A,\,X_2\in B)=E\Bigl\{1_A(X_1)\,P(X_2\in B\mid\mathcal{F}_1)\Bigr\}
\\=E\Bigl\{1_A(X_1)\,\frac{\theta\nu(B)+K(X_1)(B)}{1+\theta}\Bigr\}
\\=\frac{\theta}{1+\theta}\,\nu(B)\,\nu(A)\,+\,\frac{1}{1+\theta}\,\int_AK(y)(B)\,\nu(dy).
\end{gather*}
Therefore,
\begin{gather*}
\text{condition \eqref{d5yh8q}}\quad\Longleftrightarrow\quad (X_1,X_2)\sim (X_2,X_1).
\end{gather*}
Similarly, under \eqref{ct6h}, one obtains
\begin{gather*}
P\bigl(X_2\in A,\,X_3\in B\mid\mathcal{F}_1\bigr)=E\Bigl\{1_A(X_2)\,P(X_3\in B\mid\mathcal{F}_2)\mid\mathcal{F}_1\Bigr\}
\\=E\Bigl\{1_A(X_2)\,\frac{\theta\nu(B)+K(X_1)(B)+K(X_2)(B)}{2+\theta}\mid\mathcal{F}_1\Bigr\}
\\=\frac{1+\theta}{2+\theta}\,P(X_2\in B\mid\mathcal{F}_1)\,P(X_2\in A\mid\mathcal{F}_1)\,+\,\frac{1}{2+\theta}\,E\Bigl\{1_A(X_2)\,K(X_2)(B)\mid\mathcal{F}_1\Bigr\}\quad\text{a.s.}
\end{gather*}
and
\begin{gather*}
E\Bigl\{1_A(X_2)\,K(X_2)(B)\mid\mathcal{F}_1\Bigr\}=\frac{\theta}{1+\theta}\,\int_AK(y)(B)\,\nu(dy)+\frac{1}{1+\theta}\,\int_AK(y)(B)\,K(X_1)(dy)\quad\text{a.s.}
\end{gather*}

\vspace{0.2cm}

Next, if $X$ is exchangeable, condition \eqref{d5yh8q} follows from $(X_1,X_2)\sim (X_2,X_1)$. Moreover, $P\bigl(X_2\in A,\,X_3\in B\mid\mathcal{F}_1\bigr)=P\bigl(X_2\in B,\,X_3\in A\mid\mathcal{F}_1\bigr)$ a.s. implies
\begin{gather*}
E\Bigl\{1_A(X_2)\,K(X_2)(B)\mid\mathcal{F}_1\Bigr\}=E\Bigl\{1_B(X_2)\,K(X_2)(A)\mid\mathcal{F}_1\Bigr\}\quad\text{a.s.}
\end{gather*}
Therefore, \eqref{q2a3x} follows from \eqref{d5yh8q} and the above condition.

\vspace{0.2cm}

Conversely, assume conditions \eqref{d5yh8q}-\eqref{q2a3x}. Define
\begin{gather*}
\sigma_n(x)=\frac{\theta\nu+\sum_{i=1}^nK(x_i)}{n+\theta}\quad\quad\text{for all }n\ge 1\text{ and }x=(x_1,\ldots,x_n)\in S^n.
\end{gather*}
By \eqref{ct6h}, $P(X_{n+1}\in\cdot\mid\mathcal{F}_n)=\sigma_n(X_1,\ldots,X_n)$ a.s. Moreover, $\sigma_n(x)=\sigma_n(\pi(x))$ for all $n\ge 2$, all permutations $\pi$ of $S^n$ and all $x\in S^n$. Hence, by Theorem \ref{markov}, it suffices to show that
\begin{gather*}
P\bigl[(X_{n+1},X_{n+2})\in\cdot\mid\mathcal{F}_n\bigr]=P\bigl[(X_{n+2},X_{n+1})\in\cdot\mid\mathcal{F}_n\bigr]\quad\quad\text{a.s. for all }n\ge 0.
\end{gather*}
For $n=0$, the above condition is equivalent to \eqref{d5yh8q} (recall that $\mathcal{F}_0$ is the trivial $\sigma$-field). Therefore, it is enough to show that
\begin{gather}\label{iu7yt5r}
\int_A\sigma_{n+1}(x,y)(B)\,\sigma_n(x)(dy)=\int_B\sigma_{n+1}(x,y)(A)\,\sigma_n(x)(dy)
\end{gather}
for all $n\ge 1$, all $A,\,B\in\mathcal{B}$ and almost all $x\in S^n$ (where ``almost all" refers to the marginal distribution of $(X_1,\ldots,X_n)$).

Fix $m\ge 1$ and $A\in\mathcal{B}$. If $X_i\sim\nu$ for $i=1,\ldots,m$, then
\begin{gather*}
E\Bigl\{K(X_i)(A)\Bigr\}=\int K(y)(A)\,\nu(dy)=\nu(A)\quad\text{for }i=1,\ldots,m,
\end{gather*}
where the second equality is by \eqref{d5yh8q} (applied with $B=S$). Hence,
\begin{gather*}
P(X_{m+1}\in A)=E\Bigl\{P(X_{m+1}\in A\mid\mathcal{F}_m)\Bigr\}=\frac{\theta\nu(A)}{m+\theta}+\frac{\sum_{i=1}^mE\Bigl\{K(X_i)(A)\Bigr\}}{m+\theta}=\nu(A).
\end{gather*}
By induction, it follows that $X_i\sim\nu$ for all $i\ge 1$.

Finally, fix $n\ge 1$ and $A,\,B\in\mathcal{B}$. By \eqref{q2a3x}, there is a set $M\in\mathcal{B}$ such that $\nu(M)=1$ and
\begin{gather*}
\int_AK(y)(B)\,K(x)(dy)=\int_BK(y)(A)\,K(x)(dy)\quad\quad\text{for all }x\in M.
\end{gather*}
Thanks to this fact and condition \eqref{d5yh8q}, if $x=(x_1,\ldots,x_n)\in M^n$, one obtains
\begin{gather*}
\int_AK(y)(B)\,\sigma_n(x)(dy)=\frac{\theta\,\int_AK(y)(B)\,\nu(dy)+\sum_{i=1}^n\int_AK(y)(B)\,K(x_i)(dy)}{n+\theta}
\\=\frac{\theta\,\int_BK(y)(A)\,\nu(dy)+\sum_{i=1}^n\int_BK(y)(A)\,K(x_i)(dy)}{n+\theta}=\int_BK(y)(A)\,\sigma_n(x)(dy).
\end{gather*}
It follows that
\begin{gather*}
\int_A\sigma_{n+1}(x,y)(B)\,\sigma_n(x)(dy)=\int_A\,\frac{\theta\nu(B)+\sum_{i=1}^nK(x_i)(B)+K(y)(B)}{n+1+\theta}\,\sigma_n(x)(dy)
\\=\frac{n+\theta}{n+1+\theta}\,\sigma_n(x)(B)\,\sigma_n(x)(A)\,+\,\frac{\int_AK(y)(B)\,\sigma_n(x)(dy)}{n+1+\theta}
\\=\frac{n+\theta}{n+1+\theta}\,\sigma_n(x)(B)\,\sigma_n(x)(A)\,+\,\frac{\int_BK(y)(A)\,\sigma_n(x)(dy)}{n+1+\theta}
\\=\int_B\sigma_{n+1}(x,y)(A)\,\sigma_n(x)(dy).
\end{gather*}
Therefore, equation \eqref{iu7yt5r} holds for each $x\in M^n$. To conclude the proof, it suffices to note that, since $\nu(M)=1$ and $X_i\sim\nu$ for all $i$,
\begin{gather*}
P\bigl((X_1,\ldots,X_n)\in M^n\bigr)=1.
\end{gather*}
\end{proof}

\vspace{0.2cm}

In view of Theorem \ref{bn7}, $X$ is a kernel based Dirichlet sequence, as defined in Subsection \ref{v5ty7}, if and only if conditions \eqref{ct6h} and \eqref{d5yh8q}-\eqref{q2a3x} hold. Since \eqref{v56yh} $\Rightarrow$ \eqref{d5yh8q}-\eqref{q2a3x} (because of Lemma \ref{w5v7uj}), a sufficient condition for $X$ to be a kernel based Dirichlet sequence is that $X\in\mathcal{L}$. We do not know whether \eqref{d5yh8q}-\eqref{q2a3x} $\Rightarrow$ \eqref{v56yh}. In the sequel, however, we always assume $X\in\mathcal{L}$, namely, we always assume conditions \eqref{ct6h}-\eqref{v56yh}.

\vspace{0.2cm}

The next step is to develop some theory for $\mathcal{L}$. To this end, the following result is useful.

\begin{thm}\label{aw4rfv8}
If $X\in\mathcal{L}$, the sequence $\bigl(K(X_n):n\ge 1\bigr)$ is a Dirichlet sequence with values in $(\mathcal{P},\mathcal{C})$ and parameter the image measure $\theta\,\nu\circ K^{-1}$.
\end{thm}

\begin{proof}
By Lemma \ref{w3d5z5n9i}, there is a set $F\in\sigma(K)$ such that
\begin{gather*}
\nu(F)=1\quad\text{and}\quad K(x)(B)=\delta_x(B)\quad\quad\text{for all }B\in\sigma(K)\text{ and }x\in F.
\end{gather*}
Since $P(X_n\in F)=\nu(F)=1$ for all $n$, it follows that
\begin{gather*}
P\bigl(X_{n+1}\in\ B\mid\mathcal{F}_n\bigr)=\frac{\theta\nu(B)+\sum_{i=1}^n\delta_{X_i}(B)}{n+\theta}\quad\quad\text{for all }B\in\sigma(K)\,\text{ a.s.}
\end{gather*}

Having noted this fact, define
\begin{gather*}
\mathcal{K}_n=\sigma\bigl[K(X_1),\ldots,K(X_n)\bigr].
\end{gather*}
Since $\mathcal{K}_n\subset\mathcal{F}_n$ and $P\bigl(X_{n+1}\in\cdot\mid\mathcal{F}_n\bigr)$ is $\mathcal{K}_n$-measurable,
\begin{gather*}
P\bigl(X_{n+1}\in\cdot\mid\mathcal{K}_n\bigr)=P\bigl(X_{n+1}\in\cdot\mid\mathcal{F}_n\bigr)\quad\quad\text{a.s.}
\end{gather*}
Finally, fix $C\in\mathcal{C}$ and define $B=\{K\in C\}$. Since $B\in\sigma(K)$, one obtains
\begin{gather*}
P\Bigl[K(X_{n+1})\in C\mid\mathcal{K}_n\Bigr]=P\Bigl(X_{n+1}\in B\mid\mathcal{K}_n\Bigr)=P\Bigl(X_{n+1}\in B\mid\mathcal{F}_n\Bigr)
\\=\frac{\theta\nu(B)+\sum_{i=1}^n\delta_{X_i}(B)}{n+\theta}=\frac{\theta\nu\circ K^{-1}(C)+\sum_{i=1}^n\delta_{K(X_i)}(C)}{n+\theta}\quad\quad\text{ a.s.}
\end{gather*}
\end{proof}

\vspace{0.2cm}

We next turn to a Sethuraman-like representation for $\mathcal{L}$. Let $\mu^*$ be the random probability measure on $(S,\mathcal{B})$ defined as
\begin{gather*}
\mu^*=\sum_jV_j\,K(Z_j),
\end{gather*}
where $(Z_j)$ and $(V_j)$ are independent sequences, $(Z_j)$ is i.i.d. with $Z_1\sim\nu$, and $(V_j)$ has the stick-breaking distribution with parameter $\theta$; see Section \ref{m98f4}.

\begin{thm}\label{t67v}
If $X\in\mathcal{L}$, then
\begin{gather*}
P(\mu\in C)=P(\mu^*\in C)\quad\quad\text{for all }C\in\mathcal{C}.
\end{gather*}
\end{thm}

\begin{proof}
Let $\mu_0$ and $\mu^*_0$ be the restrictions of $\mu$ and $\mu^*$ on $\sigma(K)$. Then, $\mu_0\sim\mu^*_0$ by \cite{SET} and since $\bigl(K(X_n):n\ge 1\bigr)$ is a classical Dirichlet sequence. Hence,
\begin{gather*}
\Bigl(\mu(g_1),\ldots,\mu(g_k)\Bigr)\sim \Bigl(\mu^*(g_1),\ldots,\mu^*(g_k)\Bigr)
\end{gather*}
whenever $g_1,\ldots,g_k:S\rightarrow\mathbb{R}$ are bounded and $\sigma(K)$-measurable. In addition, for fixed $A\in\mathcal{B}$, one obtains
\begin{gather*}
\int K(x)(A)\,\mu(dx)=\lim_n\frac{\sum_{i=1}^nK(X_i)(A)}{n}=\lim_nP(X_{n+1}\in A\mid\mathcal{F}_n)=\mu(A)\quad\quad\text{a.s.}
\end{gather*}
Similarly, Lemma \ref{w5v7uj} (applied with $B=S$) implies
\begin{gather*}
\int K(x)(A)\,K(Z_j)(dx)=K(Z_j)(A)\quad\quad\text{a.s. for all }j\ge 1.
\end{gather*}
Thus,
\begin{gather*}
\int K(x)(A)\,\mu^*(dx)=\sum_jV_j\,\int K(x)(A)\,K(Z_j)(dx)
\\=\sum_jV_j\,K(Z_j)(A)=\mu^*(A)\quad\quad\text{a.s.}
\end{gather*}
Having noted these facts, fix $k\ge 1$, $A_1,\ldots,A_k\in\mathcal{B}$, and define $g_i(x)=K(x)(A_i)$ for all $x\in S$ and $i=1,\ldots,k$. Then,
\begin{gather*}
\Bigl(\mu(A_1),\ldots,\mu(A_k)\Bigr)\overset{a.s.}=\Bigl(\mu(g_1),\ldots,\mu(g_k)\Bigr)\,\sim\,\Bigl(\mu^*(g_1),\ldots,\mu^*(g_k)\Bigr)\overset{a.s.}=\Bigl(\mu^*(A_1),\ldots,\mu^*(A_k)\Bigr).
\end{gather*}
This concludes the proof.
\end{proof}

\vspace{0.2cm}

Theorem \ref{t67v} plays for $\mathcal{L}$ the same role played by \cite{SET} for $\mathcal{L}_0$. Among other things, it provides a simple way to approximate the probability distribution of $\mu$ and to obtain its posterior distribution; see forthcoming Theorem \ref{om55d} and its proof. For a further implication, define
\begin{gather*}
D_1=\bigl\{p\in\mathcal{P}:p\text{ discrete}\bigr\},\quad D_2=\bigl\{p\in\mathcal{P}:p\text{ non-atomic}\bigr\},\quad D_3=\bigl\{p\in\mathcal{P}:p\ll\nu\bigr\}.
\end{gather*}
Then, Theorem \ref{t67v} implies the following result.

\begin{thm}\label{i92xf}
If $j\in\{1,2,3\}$ and $X\in\mathcal{L}$, then $P(\mu\in D_j)\in\{0,1\}$ and
\begin{gather*}
P(\mu\in D_j)=1\quad\Leftrightarrow\quad K(x)\in D_j\text{ for }\nu\text{-almost all }x\in S.
\end{gather*}
In addition,
\begin{gather}\label{ki0z}
\sup_{A\in\mathcal{B}}\,\Abs{P\bigl(X_{n+1}\in A\mid\mathcal{F}_n\bigr)-\mu(A)}\overset{a.s.}\longrightarrow 0\quad\quad\text{as }n\rightarrow\infty.
\end{gather}
\end{thm}

\begin{proof}
Fix $j\in\{1,2,3\}$ and define $a_j=\nu\{x:K(x)\in D_j\}$. If $a_j=1$, Theorem \ref{t67v} yields
\begin{gather*}
P(\mu\in D_j)=P(\mu^*\in D_j)=P\bigl(K(Z_i)\in D_j\text{ for all }i\ge 1\bigr)=1.
\end{gather*}
Similarly, if $a_j<1$,
\begin{gather*}
P(\mu\in D_j)\le P\bigl(K(Z_i)\in D_j\text{ for }1\le i\le n\bigr)=a_j^n\longrightarrow 0\quad\quad\text{as }n\rightarrow\infty.
\end{gather*}

It remains to prove \eqref{ki0z}. Define the random probability measure
\begin{gather*}
\lambda_n=\frac{1}{n}\,\sum_{i=1}^nK(X_i).
\end{gather*}
To prove \eqref{ki0z}, it is enough to show that $\lim_n\sup_{A\in\mathcal{B}}\,\abs{\lambda_n(A)-\mu(A)}\overset{a.s.}=0$ and this limit relation is actually true if $X\in\mathcal{L}_0$; see e.g. \cite[Prop. 11]{PIT1996}. Hence, since $\bigl(K(X_n):n\ge 1\bigr)$ is a classical Dirichlet sequence, one obtains
\begin{gather*}
\sup_{A\in\sigma(K)}\,\abs{\lambda_n(A)-\mu(A)}\overset{a.s.}\longrightarrow 0.
\end{gather*}
Now, we argue as in the proof of Theorem \ref{t67v}. Precisely, for each $A\in\mathcal{B}$, Lemma \ref{w5v7uj} (applied with $B=S$) yields
\begin{gather*}
\int K(x)(A)\,\lambda_n(dx)=\frac{1}{n}\,\sum_{i=1}^n\int K(x)(A)\,K(X_i)(dx)=\frac{1}{n}\,\sum_{i=1}^nK(X_i)(A)=\lambda_n(A)\quad\text{a.s.}
\end{gather*}
Similarly, $\int K(x)(A)\,\mu(dx)=\mu(A)$ a.s. Therefore, after fixing a countable field $\mathcal{B}_0$ such that $\mathcal{B}=\sigma(\mathcal{B}_0)$, one finally obtains
\begin{gather*}
\sup_{A\in\mathcal{B}}\,\abs{\lambda_n(A)-\mu(A)}=\sup_{A\in\mathcal{B}_0}\,\abs{\lambda_n(A)-\mu(A)}
\\\overset{a.s.}=\sup_{A\in\mathcal{B}_0}\,\Abs{\int K(x)(A)\,\lambda_n(dx)-\int K(x)(A)\,\mu(dx)}
\\\le\sup_{A\in\sigma(K)}\,\abs{\lambda_n(A)-\mu(A)}\overset{a.s.}\longrightarrow 0.
\end{gather*}
\end{proof}

\vspace{0.2cm}

It is worth noting that, for an arbitrary exchangeable sequence $X$, convergence in total variation of $P\bigl(X_{n+1}\in\cdot\mid\mathcal{F}_n\bigr)$ is not guaranteed; see e.g. \cite{BPR2013}.

\vspace{0.2cm}

A further consequence of Theorem \ref{t67v} is a stable CLT (stable convergence is briefly recalled in Section \ref{m98f4}). For each $y\in\mathbb{R}^p$, let $y_i$ denote the $i$-th coordinate of $y$.

\begin{thm}\label{c9mj5}
Let $S=\mathbb{R}^p$ and $X\in\mathcal{L}$. Suppose $\int\norm{x}^2\,\nu(dx)<\infty$, where $\norm{\cdot}$ is the Euclidean norm, and
\begin{gather*}
\int y_i\,K(x)(dy)=0\quad\quad\text{for all }x\in\mathbb{R}^p\text{ and }i=1,\ldots,p.
\end{gather*}
Then,
\begin{gather*}
\frac{\sum_{i=1}^nX_i}{\sqrt{n}}\,\overset{stably}\longrightarrow\,\mathcal{N}_p(0,\Sigma)\quad\quad\text{as }n\rightarrow\infty,
\end{gather*}
where $\Sigma$ is the random covariance matrix
\begin{gather*}
\Sigma=\left(\,\int y_i\,y_j\,\mu(dy):1\le i,\,j\le p\right).
\end{gather*}
\end{thm}

\begin{proof} By standard arguments, it suffices to show that
\begin{gather*}
\frac{\sum_{i=1}^nb'X_i}{\sqrt{n}}\overset{stably}\longrightarrow\mathcal{N}_1(0,\,b'\Sigma\,b)\quad\quad\text{for each }b\in\mathbb{R}^p,
\end{gather*}
where points of $\mathbb{R}^p$ are regarded as column vectors and $b'$ denotes the transpose of $b$. Define
\begin{gather*}
\sigma_b^2=E\bigl[(b'X_1)^2\mid\mu\bigr]-E(b'X_1\mid\mu)^2.
\end{gather*}
For fixed $b\in\mathbb{R}^p$, one obtains
\begin{gather*}
n^{-1/2}\,\sum_{i=1}^n\Bigl\{b'X_i-E(b'X_1\mid\mu)\Bigr\}\overset{stably}\longrightarrow\mathcal{N}_1(0,\sigma_b^2);
\end{gather*}
see e.g. \cite[Th. 3.1]{BPR2004} and the subsequent remark. Furthermore,
\begin{gather*}
E(b'X_1\mid\mu)=\int (b'y)\,\mu(dy)=\sum_{i=1}^pb_i\int y_i\,\mu(dy)\quad\quad\quad\text{a.s. and}
\\E\bigl[(b'X_1)^2\mid\mu\bigr]=\int (b'y)^2\,\mu(dy)=\sum_{i=1}^p\sum_{j=1}^pb_ib_j\,\int y_i\,y_j\,\mu(dy)=b'\Sigma\,b\quad\text{a.s.}
\end{gather*}
Hence, it suffices to show that $\int y_i\,\mu(dy)\overset{a.s.}=0$ for all $i$, and this follows from Theorem \ref{t67v}. In fact, $\int y_i\,\mu(dy)\sim\int y_i\,\mu^*(dy)$ and
\begin{gather*}
\int y_i\,\mu^*(dy)=\sum_jV_j\,\int y_i\,K(Z_j)(dy)=0.
\end{gather*}
This concludes the proof.
\end{proof}

\vspace{0.2cm}

Theorem \ref{c9mj5} applies to Examples \ref{u7z2} and \ref{q2q}. In fact, in Example \ref{u7z2}, one has $p=1$ and $K(x)=(\delta_x+\delta_{-x})/2$. Hence, $\int y\,K(x)(dy)=0$ for all $x\in\mathbb{R}$. Example \ref{q2q} is discussed below. Here, we give conditions for convergence in total variation of $n^{-1/2}\sum_{i=1}^nX_i$.

\begin{thm}\label{t6n9ij}
In addition to the conditions of Theorem \ref{c9mj5}, suppose that $K(x)$ is not singular, with respect to Lebesgue measure, for $\nu$-almost all $x\in\mathbb{R}^p$.
Define
\begin{gather*}
Y_n=n^{-1/2}\sum_{i=1}^nX_i\quad\text{and}\quad\lambda(A)=E\bigl\{\mathcal{N}_p(0,\Sigma^*)(A)\bigr\}
\\\text{for all }A\in\mathcal{B},\text{ where}\quad\Sigma^*=\left(\,\int y_i\,y_j\,\mu^*(dy):1\le i,\,j\le p\right).
\end{gather*}
Then,
\begin{gather*}
\lim_n\,\sup_{A\in\mathcal{B}}\,\Abs{P(Y_n\in A)-\lambda(A)}=0.
\end{gather*}
\end{thm}

\begin{proof}
Let $D$ be the collection of elements of $\mathcal{P}$ which are not singular with respect to Lebesgue measure. By Theorem \ref{t67v}, $P(\mu\in D)=P(\mu^*\in D)=1$. Hence, conditionally on $\mu$, the sequence $X$ is i.i.d. and the common distribution $\mu$ belongs to $D$ a.s. Arguing as in Theorem \ref{c9mj5}, one also obtains $\int y_i\,\mu(dy)=0$ and $\int\norm{y}^2\mu(dy)<\infty$ a.s. for all $i$. Thus, conditionally on $\mu$, $Y_n$ converges to $\mathcal{N}_p(0,\Sigma)$ in total variation (see e.g. \cite{BACA}) that is
\begin{gather*}
\sup_{A\in\mathcal{B}}\,\Abs{P(Y_n\in A\mid\mu)-\mathcal{N}_p(0,\Sigma)(A)}\overset{a.s.}\longrightarrow 0.
\end{gather*}
Finally, $\Sigma\sim\Sigma^*$ implies $\lambda(\cdot)=E\bigl\{\mathcal{N}_p(0,\Sigma)(\cdot)\bigr\}$. Hence,
\begin{gather*}
\sup_{A\in\mathcal{B}}\,\Abs{P(Y_n\in A)-\lambda(A)}=\sup_{A\in\mathcal{B}}\,\Abs{P(Y_n\in A)-E\bigl\{\mathcal{N}_p(0,\Sigma)(A)\bigr\}}
\\\le E\left\{\sup_{A\in\mathcal{B}}\,\Abs{P(Y_n\in A\mid\mu)-\mathcal{N}_p(0,\Sigma)(A)}\right\}\,\longrightarrow\,0\quad\quad\text{as }n\rightarrow\infty.
\end{gather*}
\end{proof}

Our last result deals with the posterior distribution of $\mu$. We aim to find the conditional distribution of $\mu$ given $\mathcal{F}_n=\sigma(X_1,\ldots,X_n)$. To this end, for each $n\ge 1$, we denote by
\begin{gather*}
V^{(n)}=\bigl(V_j^{(n)}:j\ge 1\bigr)\quad\text{and}\quad Z^{(n)}=\bigl(Z_j^{(n)}:j\ge 1\bigr)
\end{gather*}
two sequences such that:

\vspace{0.2cm}

\begin{itemize}

\item[(i)] $V^{(n)}$ and $Z^{(n)}$ are conditionally independent given $\mathcal{F}_n$;

\vspace{0.2cm}

\item[(ii)] $V^{(n)}$ has the stick-breaking distribution with parameter $n+\theta$ conditionally on $\mathcal{F}_n$;

\vspace{0.2cm}

\item[(iii)] $Z^{(n)}$ is i.i.d. conditionally on $\mathcal{F}_n$ with
\begin{gather*}
P(Z_1^{(n)}\in\cdot\mid\mathcal{F}_n)=P\bigl(X_{n+1}\in\cdot\mid\mathcal{F}_n\bigr)=\frac{\theta\nu(\cdot)+\sum_{i=1}^nK(X_i)(\cdot)}{n+\theta}\quad\quad\text{a.s.}
\end{gather*}

\end{itemize}

\vspace{0.2cm}

\noindent Moreover, we let
\begin{gather*}
\mu_n^*=\sum_jV_j^{(n)}\,K\bigl(Z_j^{(n)}\bigr).
\end{gather*}

\begin{thm}\label{om55d}
If $X\in\mathcal{L}$, then
\begin{gather*}
P(\mu\in C\mid\mathcal{F}_n)=P(\mu_n^*\in C\mid\mathcal{F}_n)\quad\quad\text{a.s. for all }C\in\mathcal{C}\text{ and }n\ge 1.
\end{gather*}
\end{thm}

\vspace{0.2cm}

We recall that, if $X\in\mathcal{L}_0$ and $X$ has parameter $\theta\nu$ (i.e., if $K=\delta$) then
\begin{gather*}
P(\mu\in C\mid\mathcal{F}_n)=\mathcal{D}\Bigl(\theta\nu+\sum_{i=1}^n\delta_{X_i}\Bigr)(C)=P(\mu_n^*\in C\mid\mathcal{F}_n)\quad\quad\text{a.s.}
\end{gather*}
Hence, Theorem \ref{om55d} extends to $\mathcal{L}$ the conjugacy property of $\mathcal{L}_0$. Such a property is clearly useful as regards Bayesian statistical inference. On one hand, the Bayesian analysis of $X\in\mathcal{L}$ is as simple as that of $X\in\mathcal{L}_0$. On the other hand, $\mathcal{L}$ is able to model much more situations than $\mathcal{L}_0$. As an obvious example, for $X\in\mathcal{L}$, it may be that $P(X_i=X_j)=0$ if $i\ne j$. See e.g. Example \ref{w4f} and Theorem \ref{i92xf}.

\vspace{0.2cm}

Theorem \ref{om55d} can be proved in various ways. We report here the simplest and most direct proof. Such a proof relies on Theorem \ref{t67v} and the definition of $\mathcal{L}$ in terms of predictive distributions.

\vspace{0.2cm}

\begin{proof}[\textbf{Proof of Theorem \ref{om55d}}]

Throughout this proof, if $X$ satisfies conditions \eqref{ct6h}-\eqref{v56yh}, we say that $X\in\mathcal{L}$ and $X$ has {\em parameter} $(\theta\nu,\,K)$.

\vspace{0.2cm}

Fix $n\ge 1$ and define the sequence
\begin{gather*}
X^{(n)}=\bigl(X^{(n)}_i:i\ge 1\bigr)=\bigl(X_{n+i}:i\ge 1\bigr).
\end{gather*}
Define also the random measure
\begin{gather*}
J_n=\theta\nu+\sum_{i=1}^nK(X_i).
\end{gather*}
It suffices to show that, conditionally on $\mathcal{F}_n$, one obtains
\begin{gather}\label{gyy87j}
X^{(n)}\in\mathcal{L}\quad\text{and}\quad X^{(n)}\text{ has parameter }(J_n,\,K)\text{ a.s.}
\end{gather}
In fact, under \eqref{gyy87j}, Theorem \ref{t67v} implies that $\mu\sim\mu_n^*$ conditionally on $\mathcal{F}_n$, namely
\begin{gather*}
P(\mu\in\cdot\mid\mathcal{F}_n)=P(\mu_n^*\in\cdot\mid\mathcal{F}_n)\quad\quad\text{a.s.}
\end{gather*}

\vspace{0.2cm}

In turn, condition \eqref{gyy87j} follows directly from the definition. Define in fact
\begin{gather*}
P^{\mathcal{F}_n}\bigl(X^{(n)}\in B\bigr)=P\bigl(X^{(n)}\in B\mid\mathcal{F}_n\bigr)\quad\quad\text{for all }B\in\mathcal{B}^\infty\text{ a.s.}
\end{gather*}
Then,
\begin{gather*}
P^{\mathcal{F}_n}\bigl(X^{(n)}_1\in\cdot\bigr)=P\bigl(X_{n+1}\in\cdot\mid\mathcal{F}_n\bigr)=\frac{\theta\nu+\sum_{i=1}^{n}K(X_i)}{n+\theta}=\frac{J_n}{n+\theta}\quad\text{a.s.}
\end{gather*}
and
\begin{gather*}
P^{\mathcal{F}_n}\Bigl(X^{(n)}_{m+1}\in\cdot\mid X^{(n)}_1,\ldots,X^{(n)}_m\Bigr)=P\bigl(X_{n+m+1}\in\cdot\mid\mathcal{F}_{n+m}\bigr)
\\=\frac{\theta\nu+\sum_{i=1}^{n+m}K(X_i)}{n+m+\theta}=\frac{J_n+\sum_{i=1}^mK\bigl(X^{(n)}_i\bigr)}{(n+\theta)+m}\quad\text{a.s. for all }m\ge 1.
\end{gather*}
This concludes the proof.
\end{proof}

\vspace{0.2cm}

\section{Open problems and examples}

This section is split into two parts. First, we discuss some hints for future research and then we give three further examples.

\begin{itemize}

\item\textbf{An enlargment of $\mathcal{L}$.} The class $\mathcal{L}$ could be made larger. In this case, however, some of the basic properties of $\mathcal{L}_0$ would be lost. As an example, suppose that
\begin{gather*}
X_1\sim\nu\quad\text{and}\quad P\bigl(X_{n+1}\in\cdot\mid\mathcal{F}_n\bigr)=c_n\,\nu(\cdot)+(1-c_n)\,\,\frac{\sum_{i=1}^nK(X_i)(\cdot)}{n}\quad\quad\text{a.s.,}
\end{gather*}
where the kernel $K$ satisfies condition \eqref{v56yh} and $c_n\in [0,1]$ is a constant. To make $X$ closer to $\mathcal{L}_0$, suppose also that $\lim_nc_n=0$. Then, $X$ is exchangeable and $X\in\mathcal{L}$ provided $c_n=\theta/(n+\theta)$. Furthermore, various properties of $\mathcal{L}_0$ are preserved, including $\mu\sim\sum_jV_j\,K(Z_j)$ where $(V_j)$ and $(Z_j)$ are independent sequences and $(Z_j)$ is i.i.d. with $Z_1\sim\nu$. Unlike Theorem \ref{t67v}, however, the probability distribution of $(V_j)$ is unknown (to us). Similarly, we do not know whether some form of Theorem \ref{om55d} is still valid.

\vspace{0.2cm}

\item\textbf{A characterization of $\mathcal{L}$.} Denoting by $\mathcal{L}^*$ the class of kernel based Dirichlet sequences, it is tempting to conjecture that $\mathcal{L}=\mathcal{L}^*$. Since $\mathcal{L}\subset\mathcal{L}^*$, the question is whether there is an exchangeable sequence satisfying condition \eqref{ct6h} but not condition \eqref{v56yh}. Lemma \ref{w5v7uj} and Theorem \ref{bn7} may be useful to address this issue.

\vspace{0.2cm}

\item\textbf{Self-similarity.} Suppose $X\in\mathcal{L}_0$ and take $A\in\mathcal{B}$ such that $0<\nu(A)<1$. Then, the distribution of the random probability measure $\mu(\cdot\mid A)$ is still of the Dirichlet type with $\nu$ and $\theta$ replaced by $\nu(\cdot\mid A)$ and $\theta\,\nu(A)$, respectively. In addition, $\mu(A)$, $\mu(\cdot\mid A)$ and $\mu(\cdot\mid A^c)$ are independent random elements; see \cite[p. 61]{GHOSVANDER}. A question is whether this property of $\mathcal{L}_0$, called {\em self-similarity}, is still true for $\mathcal{L}$. Suppose $X\in\mathcal{L}$ and $K$ is a r.c.d. for $\nu$ given the sub-$\sigma$-field $\mathcal{G}\subset\mathcal{B}$. If $A\in\mathcal{G}$, then $K(X_i)(A)=1_A(X_i)$ a.s. for all $i$. Based on this fact, $\mu(\cdot\mid A)$ can be shown to have the same distribution as $\mu$ with $\nu$ and $\theta$ replaced by $\nu(\cdot\mid A)$ and $\theta\,\nu(A)$. Hence, $\mathcal{L}$ satisfies some form of self-similarity when $A\in\mathcal{G}$. However, we do not know whether $\mu(A)$, $\mu(\cdot\mid A)$ and $\mu(\cdot\mid A^c)$ are independent. Similarly, we do not know what happens if $A\notin\mathcal{G}$.

\vspace{0.2cm}

\item\textbf{Topological support.} The topological support of a Borel probability $\lambda$ on a separable metric space, denoted $\mathcal{S}(\lambda)$, is the smallest closed set $A$ satisfying $\lambda(A)=1$. Let $\mathcal{P}$ be equipped with the topology of weak convergence, i.e., the weakest topology on $\mathcal{P}$ which makes continuous the maps $p\mapsto\int f\,dp$ for all bounded continuous functions $f:S\rightarrow\mathbb{R}$. Moreover, let
\begin{gather*}
\Pi(C)=P(\mu\in C),\quad\quad C\in\mathcal{C},
\end{gather*}
be the prior corresponding to $\mu$. It is well known that
\begin{gather*}
\mathcal{S}(\Pi)=\bigl\{p\in\mathcal{P}:\mathcal{S}(p)\subset\mathcal{S}(\nu)\bigr\}
\end{gather*}
whenever $X\in\mathcal{L}_0$; see \cite{FERG} and \cite{MAJUM}. As a consequence, $\mathcal{S}(\Pi)=\mathcal{P}$ if $\mathcal{S}(\nu)=S$. A (natural) question is whether, under some conditions on $K$, this basic property of $\mathcal{L}_0$ is preserved by $\mathcal{L}$. The next result provides a partial answer.

\begin{prop}
If $X\in\mathcal{L}$ and $\mathcal{S}(\Pi)=\mathcal{P}$, then
\begin{gather}\label{hg9im5}
\nu\bigl\{x\in S:K(x)(A)\le u\bigr\}< 1
\end{gather}
for all $u<1$ and all non-empty open sets $A\subset S$.
\end{prop}

\begin{proof}
First note that $\mathcal{S}(\Pi)=\mathcal{P}$ if and only if $\Pi(U)>0$ for each non-empty open set $U\subset\mathcal{P}$. Having noted this fact, suppose $\nu\bigl\{x\in S:K(x)(A)\le u\bigr\}=1$, for some $u<1$ and some non-empty open set $A\subset S$, and define
\begin{gather*}
U=\bigl\{p\in\mathcal{P}:p(A)>u\bigr\}.
\end{gather*}
Then, $U$ is open and non-empty. Moreover, if $V_j$, $Z_j$ and $\mu^*$ are as in Section \ref{r1e2s3}, one obtains $K(Z_j)(A)\le u$ for all $j$ a.s. and
\begin{gather*}
\mu^*(A)=\sum_jV_j\,K(Z_j)(A)\le u\,\sum_jV_j=u\quad\quad\text{a.s.}
\end{gather*}
By Theorem \ref{t67v}, it follows that
\begin{gather*}
\Pi(U)=P\bigl(\mu(A)>u\bigr)=P\bigl(\mu^*(A)>u\bigr)=0.
\end{gather*}
Hence, $\mathcal{S}(\Pi)$ is a proper subset of $\mathcal{P}$.
\end{proof}
Possibly, some version of condition \eqref{hg9im5} suffices for $\mathcal{S}(\Pi)=\mathcal{P}$. However, condition \eqref{hg9im5} alone suggests that $\mathcal{S}(\Pi)$ is usually a proper subset of $\mathcal{P}$. In Example \ref{u7z2}, for instance, condition \eqref{hg9im5} fails (just take $A=(0,\infty)$ and note that $K(x)(A)\le 1/2$ for all $x$). Finally, we mention here a property of $\mathcal{L}_0$ which is preserved by $\mathcal{L}$. If $X\in\mathcal{L}$ and $\mathcal{S}(\nu)=S$, the prior $\Pi$ a.s. selects probability measures with full support, i.e.
\begin{gather*}
\Pi\bigl\{p\in\mathcal{P}:\mathcal{S}(p)=S\bigr\}=1.
\end{gather*}

\end{itemize}

\vspace{0.2cm}

We next turn to examples.

\vspace{0.2cm}

\begin{ex}\label{v5e3si9} \textbf{(Example \ref{w4f} continued).}
Let $\mathcal{H}\subset\mathcal{B}$ be a countable partition of $S$ such that $\nu(H)>0$ for all $H\in\mathcal{H}$. Then, $K(x)=\nu\bigl[\cdot\mid H(x)\bigr]$ is a r.c.d. for $\nu$ given $\sigma(\mathcal{H})$, where $H(x)$ is the only $H\in\mathcal{H}$ such that $x\in H$. Therefore, $X\in\mathcal{L}$ provided $X_1\sim\nu$ and
\begin{gather*}
P\bigl(X_{n+1}\in\cdot\mid\mathcal{F}_n\bigr)=\frac{\theta\nu(\cdot)+\sum_{i=1}^n\nu\bigl[\cdot\mid H(X_i)\bigr]}{n+\theta}\quad\quad\text{a.s.}
\end{gather*}

In this example, for each $A\in\mathcal{B}$, one obtains
\begin{gather*}
\mu(A)=\lim_nP\bigl(X_{n+1}\in A\mid\mathcal{F}_n\bigr)=\sum_{H\in\mathcal{H}}\mu(H)\,\nu(A\mid H)\quad\text{a.s.}
\end{gather*}
where $\mu(H)\overset{a.s.}=\lim_n(1/n)\,\sum_{i=1}^n1_H(X_i)$. To grasp further information about $\mu$, define
\begin{gather*}
b(H)=\sum_jV_j\,1_H(Z_j),\quad\quad H\in\mathcal{H},
\end{gather*}
where $(V_j)$ and $(Z_j)$ are independent, $(Z_j)$ is i.i.d. with $Z_1\sim\nu$, and $(V_j)$ has the stick breaking distribution with parameter $\theta$. Then, Theorem \ref{t67v} yields
\begin{gather*}
\mu\sim\mu^*=\sum_{H\in\mathcal{H}}b(H)\,\nu(\cdot\mid H).
\end{gather*}
Therefore,
\begin{gather*}
\bigl(\mu(H):H\in\mathcal{H}\bigr)\sim \bigl(\mu^*(H):H\in\mathcal{H}\bigr)=\bigl(b(H):H\in\mathcal{H}\bigr).
\end{gather*}

To evaluate the posterior distribution of $\mu$, fix $n\ge 1$ and take two sequences $V^{(n)}=\bigl(V_j^{(n)}:j\ge 1\bigr)$ and $Z^{(n)}=\bigl(Z_j^{(n)}:j\ge 1\bigr)$ satisfying conditions (i)-(ii)-(iii). Recall that, by (iii), $Z^{(n)}$ is i.i.d. conditionally on $\mathcal{F}_n$ with
\begin{gather*}
P(Z_1^{(n)}\in\cdot\mid\mathcal{F}_n)=P\bigl(X_{n+1}\in\cdot\mid\mathcal{F}_n\bigr)\quad\quad\text{a.s.}
\end{gather*}
Define
\begin{gather*}
b_n(H)=\sum_jV_j^{(n)}\,1_H\bigl(Z_j^{(n)}\bigr)\quad\text{and}\quad\mu_n^*=\sum_{H\in\mathcal{H}}b_n(H)\,\nu(\cdot\mid H).
\end{gather*}
Then, Theorem \ref{om55d} implies $\mu\sim\mu_n^*$ conditionally on $\mathcal{F}_n$.
\end{ex}

\vspace{0.2cm}

\begin{ex}\label{q2q}
Let $\norm{\cdot}$ be the Euclidean norm on $S=\mathbb{R}^p$. For $t\ge 0$, let $\mathcal{U}_t\in\mathcal{P}$ be uniform on the spherical surface $\{x:\norm{x}=t\}$ (with $\mathcal{U}_0=\delta_0$) and
\begin{gather*}
\nu(A)=\int_0^\infty \mathcal{U}_t(A)\,e^{-t}\,dt\quad\quad\text{for all }A\in\mathcal{B}.
\end{gather*}
Then, $K(x)=\mathcal{U}_{\norm{x}}$ is a r.c.d. for $\nu$ given $\sigma(\norm{\cdot})$. Hence, $X\in\mathcal{L}$ whenever $X_1\sim\nu$ and
\begin{gather*}
P\bigl(X_{n+1}\in\cdot\mid\mathcal{F}_n\bigr)=\frac{\theta\nu(\cdot)+\sum_{i=1}^n \mathcal{U}_{\norm{X_i}}(\cdot)}{n+\theta}\quad\quad\text{a.s.}
\end{gather*}

Theorem \ref{c9mj5} applies to this example. To see this, first note that
\begin{gather*}
\int\norm{x}^2\,\nu(dx)=\int_0^\infty\int\norm{x}^2\,\mathcal{U}_t(dx)\,e^{-t}dt=\int_0^\infty t^2\,e^{-t}dt<\infty.
\end{gather*}
Moreover, since $\mathcal{U}_t$ is invariant under rotations,
\begin{gather}\label{z3a4lio}
\int y_i\,\mathcal{U}_t(dy)=\int y_i\,y_j\,\mathcal{U}_t(dy)=0\quad\text{and}\quad\int y_i^2\,\mathcal{U}_t(dy)=t^2/p
\end{gather}
for all $t$, all $i$ and all $j\ne i$. (Recall that $y_i$ denotes the $i$-th coordinate of a point $y\in\mathbb{R}^p$). Because of \eqref{z3a4lio},
\begin{gather*}
\int y_i\,K(x)(dy)=\int y_i\,\mathcal{U}_{\norm{x}}(dy)=0\quad\quad\text{for all }x\in\mathbb{R}^p\text{ and }i=1,\ldots,p.
\end{gather*}
Therefore, Theorem \ref{c9mj5} yields
\begin{gather*}
\frac{\sum_{i=1}^nX_i}{\sqrt{n}}\overset{stably}\longrightarrow\mathcal{N}_p(0,\Sigma)
\end{gather*}
where $\Sigma$ is the random covariance matrix with entries
\begin{gather*}
\sigma_{ij}=\int y_i\,y_j\,\mu(dy)=\lim_n\frac{1}{n}\,\sum_{r=1}^n\int y_i\,y_j\,\mathcal{U}_{\norm{X_r}}(dy)\quad\quad\text{a.s.}
\end{gather*}
It is even possible be more precise about $\Sigma$. In fact, using \eqref{z3a4lio} again, one obtains $\sigma_{ij}=0$ for $i\ne j$ and
\begin{gather*}
\sigma_{ii}=\lim_n\frac{1}{n}\,\sum_{r=1}^n\int y_i^2\,\mathcal{U}_{\norm{X_r}}(dy)=\frac{1}{p}\,\lim_n\frac{1}{n}\,\sum_{r=1}^n\norm{X_r}^2=\frac{1}{p}\,\int\norm{x}^2\,\mu(dx)\quad\text{a.s.}
\end{gather*}
Hence, if $I$ denotes the $p\times p$ identity matrix,
\begin{gather*}
\Sigma=\sigma_{11}\,I\quad\text{where}\quad\sigma_{11}=(1/p)\,\int\norm{x}^2\,\mu(dx).
\end{gather*}
Two last remarks are in order. First, in the notation of Theorem \ref{t67v},
\begin{gather*}
\int\norm{x}^2\,\mu(dx)\sim\int\norm{x}^2\,\mu^*(dx)=\sum_jV_j\,\norm{Z_j}^2.
\end{gather*}
Second, exploiting stable convergence and $\sigma_{11}>0$ a.s., one also obtains
\begin{gather*}
\sqrt{p}\,\,\frac{\sum_{i=1}^nX_i}{\sqrt{\sum_{i=1}^n\norm{X_i}^2}}=\sqrt{p}\,\,\frac{n^{-1/2}\sum_{i=1}^nX_i}{\sqrt{n^{-1}\sum_{i=1}^n\norm{X_i}^2}}\,\overset{stably}\longrightarrow\,\mathcal{N}_p(0,I).
\end{gather*}
\end{ex}

\vspace{0.2cm}

\begin{ex}\label{w35b8uh6}
Let $F$ be a countable class of measurable maps $f:S\rightarrow S$ and
\begin{gather*}
\mathcal{I}=\bigl\{\lambda\in\mathcal{P}:\lambda=\lambda\circ f^{-1}\text{ for each }f\in F\bigr\}
\end{gather*}
the set of $F$-invariant probability measures. Let
\begin{gather*}
\mathcal{G}=\bigl\{A\in\mathcal{B}:f^{-1}(A)=A\text{ for all }f\in F\bigr\}
\end{gather*}
be the sub-$\sigma$-field of $F$-invariant measurable sets. In this example, we assume that $\nu\in\mathcal{I}$ and conditions \eqref{ct6h}-\eqref{v56yh} hold with $\mathcal{G}$ as above.

\medskip

Under these conditions, it is not hard to see that $K(x)\in \mathcal{I}$ for $\nu$-almost all $x\in S$; see e.g. \cite{MAITRA}. Hence, $P\bigl(X_{n+1}\in\cdot\mid\mathcal{F}_n\bigr)\in\mathcal{I}$ a.s. which in turn implies
\begin{gather*}
\mu(f^{-1}A)\overset{a.s.}=\lim_n P\bigl(f(X_{n+1})\in A\mid\mathcal{F}_n\bigr)\overset{a.s.}=\lim_n P\bigl(X_{n+1}\in A\mid\mathcal{F}_n\bigr)\overset{a.s.}=\mu(A)
\end{gather*}
for fixed $A\in\mathcal{B}$ and $f\in F$. Since $F$ is countable and $\mathcal{B}$ countably generated, one finally obtains
\begin{gather*}
P(\mu\in\mathcal{I})=1.
\end{gather*}
This fact is meaningful from the Bayesian point of view. It means that the prior corresponding to $\mu$ (namely, $\Pi(C)=P(\mu\in C)$ for all $C\in\mathcal{C}$)
selects $F$-invariant laws a.s. Such priors are actually useful in some practical problems; see e.g. \cite{DALAL} and \cite{HOZA}.

\medskip

Example \ref{u7z2} is a special case of the previous choice of $\mathcal{G}$. Another example, borrowed from \cite[Ex. 12]{BDPR2021}, is $S=\mathbb{R}^d$ and $F$ the class of all permutations of $\mathbb{R}^d$. In this case, $\mathcal{I}$ is the set of exchangeable probabilities on the Borel sets of $\mathbb{R}^d$. Moreover, if $\nu$ is exchangeable, $K$ can be written as
\begin{gather*}
K(x)=\frac{\sum_{\pi\in F}\delta_{\pi(x)}}{d!}\quad\quad\text{for all }x\in\mathbb{R}^d.
\end{gather*}

\medskip

A last remark is in order.

\medskip

\textbf{Claim:} If $A_1,\ldots,A_k$ is a partition of $S$ such that $A_i\in\mathcal{G}$ for all $i$, then the $k$-dimensional vector $\bigl(\mu(A_1),\ldots,\mu(A_k)\bigr)$ has Dirichlet distribution with parameters $\theta\,\nu(A_1),\ldots,\theta\,\nu(A_k)$.

\medskip

To prove the Claim, because of Theorem \ref{t67v}, it suffices to show that $\bigl(\mu^*(A_1),\ldots,\mu^*(A_k)\bigr)$ has the desired distribution. In addition, $K(x)(A_i)=1_{A_i}(x)=\delta_x(A_i)$, for $\nu$-almost all $x\in S$, since $A_i\in\mathcal{G}$ and $K$ is a r.c.d. for $\nu$ given $\mathcal{G}$. Therefore,
\begin{gather*}
\mu^*(A_i)=\sum_jV_j\,K(Z_j)(A_i)\overset{a.s.}=\sum_jV_j\,\delta_{Z_j}(A_i)\quad\quad\text{for all }i,
\end{gather*}
and this implies that $\bigl(\mu^*(A_1),\ldots,\mu^*(A_k)\bigr)$ has Dirichlet distribution with parameters $\theta\,\nu(A_1),\ldots,\theta\,\nu(A_k)$.

\medskip

In view of the Claim, $\mu$ is a Dirichlet invariant process in the sense of Definition 2 of \cite{DALAL}. Thus, arguing as above, a large class of such processes can be easily obtained. Note also that, unlike \cite{DALAL}, $F$ is not necessarily a group.
\end{ex}

\medskip

\textbf{Acknowledgments:}
This paper has been improved by the remarks and suggestions of the AE and four anonymous referees.

\end{document}